\documentclass[10pt]{article}
\usepackage{amsmath, amsthm}\usepackage{enumerate}
\usepackage{amssymb}\usepackage{bbold}
\usepackage{color}
\newtheorem{theorem}{Theorem}[section]
\newtheorem{proposition}[theorem]{Proposition}

\newtheorem{lemma}[theorem]{Lemma}
\newtheorem{example}[theorem]{Example}

\newtheorem{definition}[theorem]{Definition}

\DeclareMathOperator{\convo}{\xrightarrow[]{o}}
\DeclareMathOperator{\convso}{\xrightarrow[]{so}}
\DeclareMathOperator{\convvso}{\xrightarrow[]{vso}}

\DeclareMathOperator{\convn}{\xrightarrow[]{\|\cdot\|}}

\makeatletter
\renewcommand{\subsection}{\@startsection{subsection}{1}
{0pt}{3.25ex plus 1ex minus.2ex}{-1em}{\normalfont\normalsize\bf}}
\makeatother
\begin{document}

\title{{\bf On Levi operators between normed and vector lattices}}
\maketitle
\author{\centering{{Eduard Emelyanov$^{1}$\\ 
\small $1$ Sobolev Institute of Mathematics, Novosibirsk, Russia}

\bigskip

\abstract{The notion of a Levi operator is an operator abstraction of the Levy property of a norm or, more generally 
of the Levi topology on a locally solid vector lattice.
Various aspects of Levi operators have been studied recently by several authors. The present paper 
is devoted to Levi operators from a normed lattice to a vector lattice. It is proved that every finite rank operator 
is a Levi operator. An example is given showing that the sum of a positive rank one operator 
and a positive compact Levi operator need not to be a Levi operator.
We prove that every quasi Levi operator is continuous. 
It is shown that the set of Levi operators on the space of convergent sequences is not complete in the operator norm.
Several results concerning the domination problem for Levi operators and the relations between Levi operators and KB-spaces 
are established.


\large

\medskip

\section{Introduction}

\hspace{4mm}
A normed lattice $E$ has a Levi norm whenever every increasing 
bounded net in $E$ converges in order. An effective
line in the study of operator abstractions of a vector 
space property consists in replacing vector space $X$ by an operator $T:X\to Y$,
and redistributing the property between $X$, $Y$, and $T$
(see, for example \cite{AlEG2022,EG2023} and references therein).
Several kinds of Levi operators introduced recently \cite{AlEG2022,E2025-2,GE2022,ZC2022} as operator abstractions of the Levi norm. 
We shall investigate Levi operators following \cite{AlEG2022}.

Throughout the paper vector spaces are real, vector lattices Archimedean,
and operators linear. By $\text{\rm L}_+(E,F)$ (resp., $\text{\rm L}(E,F)$, $\text{\rm F}(E,F)$, $\text{\rm K}(E,F)$) we denote the set of 
positive (resp., continuous, finite rank continuous, compact) operators from $E$ to $F$.
Symbol $I_X$ stands for the identity operator in $X$ and $B_X$ for the closed unit ball of $X$.
By VL (NL, BL) we abbreviate a vector lattice (resp., normed lattice, Banach lattice).
Let $(y_\alpha)$ be a net in a VL $E$. We write 
\begin{enumerate}[-]
\item\
$y_\alpha\uparrow$ (resp., $y_\alpha\downarrow$) if the net $(y_\alpha)$ increases 
(resp., decreases).
\item\
$y_\alpha\downarrow 0$ whenever $y_\alpha\downarrow$ and $\inf_E y_\alpha=0$.
\end{enumerate}

\begin{definition}\label{order convergence} 
{\em
Let $(x_\alpha)$ be a net in a VL $E$ and $x\in E$. Then
\begin{enumerate}[$a)$]
\item\ 
$(x_\alpha)$ order converges to $x$ (briefly, $x_\alpha\stackrel{\text{\rm o}}{\to} x$) 
whenever there exists a net $(y_\beta)$ in $E$, $y_\beta\downarrow 0$ and, for each $\beta$ 
there is an $\alpha_\beta$ with $|x_\alpha-x|\le y_\beta$ for all $\alpha\ge\alpha_\beta$. 
\item\ 
$(x_\alpha)$ strong order converges to $x$ (briefly, $x_\alpha\stackrel{\text{\rm so}}{\to} x$) 
whenever there exist a net $y_\alpha\downarrow 0$ in $E$ and $\alpha_0$ such that $|x_\alpha-x|\le y_\alpha$ for all $\alpha\ge\alpha_0$.
\item\ 
$(x_\alpha)$ very strong order converges to $x$ ($x_\alpha\convvso x$) 
whenever there exists a net $y_\alpha\downarrow 0$ in $E$ such that $|x_\alpha-x|\le y_\alpha$ for all $\alpha$.
\end{enumerate}
}
\end{definition}
\noindent
Although the notion of order convergence as in Definition \ref{order convergence}~$a)$ is standard now-days (e.g. \cite{EGK2019}), it differs with one 
used in some textbooks like \cite{AB2006,Mey1991} that agrees with vso convergence. 
Clearly, $x_\alpha\convvso x\Longrightarrow x_\alpha\convso x\Longrightarrow x_\alpha\convo x$.
The so-conver\-gence agrees with o-convergence in a Dedekind complete VL \cite[Proposition 1.5]{AS2005}.  

\medskip
Now, we present main definition of our paper following \cite[Definition 1.1]{AlEG2022} with a new notion added in the item $b)$.

\begin{definition}\label{order-to-topology} 
{\em
An operator $T$ from a NL $E$ to a VL $F$ is
\begin{enumerate}[$a)$]
\item 
\text{\rm ($\sigma$-) Levi} if, for every bounded
increasing net (sequence) $(x_\alpha)$ of $E$,
there exists an $x\in E$ with $Tx_\alpha\stackrel{\text{\rm o}}{\to} Tx$.
(briefly, $T\in\text{\rm L}_{\text{\rm Levi}}(E,F)$ 
(resp., $T\in\text{\rm L}^\sigma_{\text{\rm Levi}}(E,F)$).
\item 
\text{\rm completely quasi ($\sigma$-) Levi} 
if $(Tx_\alpha)$ o-converges for every bounded increasing net 
(sequence) $(x_\alpha)$ of $E$. 
(briefly, $T\in\text{\rm L}_{\text{\rm cqLevi}}(E,F)$ 
(resp., $T\in\text{\rm L}^\sigma_{\text{\rm cqLevi}}(E,F)$).
\item 
\text{\rm quasi ($\sigma$-) Levi} if $T$ takes bounded
increasing nets (sequences) of $E$ to o-Cauchy ones.
(briefly, we write $T\in\text{\rm L}_{\text{\rm qLevi}}(E,F)$ 
(resp., $T\in\text{\rm L}^\sigma_{\text{\rm qLevi}}(E,F)$).
\end{enumerate}}
\end{definition}
\medskip
\noindent
It follows directly from Definition \ref{order-to-topology} that
$$
\begin{matrix}\label{matr}
   \text{\rm L}_{\text{\rm Levi}}(E,F) & \subseteq &   
   \text{\rm L}_{\text{\rm cqLevi}}(E,F)
                      & \subseteq &   \text{\rm L}_{\text{\rm qLevi}}(E,F)  \\
  \text{\small $|\bigcap$}  &  & \text{\small $|\bigcap$} &   
  \ \ \ \ \ \ \ \  & \text{\small $|\bigcap$}  \ \ \ \ \ \ \ \ \ \ \ \ \ \ \ \\
  \text{\rm L}^\sigma_{\text{\rm Levi}}(E,F) &\subseteq &    
  \text{\rm L}^\sigma_{\text{\rm cqLevi}}(E,F) & \subseteq &   
  \text{\rm L}^\sigma_{\text{\rm qLevi}}(E,F)
\end{matrix}
\eqno(*)
$$
Replacing o- by vso-, the comp\-letely quasi ($\sigma$-) Levi 
operators turn to ($\sigma$-) Levi operators of Z. Chen and F. Zhang \cite{ZC2022}.
Since vso-convergence implies order conver\-gence, each ($\sigma$-) Levi 
operator  $T:E\to F$ of \cite{ZC2022} is completely quasi ($\sigma$-) Levi in the sense of
Definition \ref{order-to-topology}. 

\medskip
The present paper is organized as follows.
Section 2 contains several basic properties of Levi operators.
Section 3 is devoted to more advanced topics, e.g.: 
to continuity of quasi Levi operators; characterization of KB-spaces as
those order continuous BLs on which every positive compact operator is Levi; and
to examples showing that the set of Levi operators in general is not closed under 
the addition, the operator norm, and passing to dominated operators. 

\medskip
We refer the reader to C. D. Aliprantis and O. Burkinshaw \cite{AB2006}, S. S. Kutateladze \cite{Kut1983}, and P. Meyer-Nieberg \cite{Mey1991} for unexplained terminology.

\section{Basic properties of Levi operators}

\hspace{4mm}
Here, we discuss basic properties of Levi operators and related examples.

\medskip
First, we show that all inclusions in $(*)$ are proper.
Let $\ell^\infty_\omega(\mathbb{R})$ be a BL of 
real-valued bounded countably supported functions on $\mathbb{R}$. 
Clearly, $I_{\ell^\infty_\omega(\mathbb{R})}\in
\text{\rm L}^\sigma_{\text{\rm Levi}}(\ell^\infty_\omega(\mathbb{R}))$. 
Take a bounded 
increasing net $(1_\alpha)_{\alpha\in{\cal P}_{fin}(\mathbb{R})}$, where
${\cal P}_{fin}(\mathbb{R})$ is the set of finite subsets of $\mathbb{R}$
directed by inclusion, and $1_\alpha$  is the indicator function of 
$\alpha$. Since $(1_\alpha)$ is not o-Cauchy, 
$I_{\ell^\infty_\omega(\mathbb{R})}\notin
\text{\rm L}_{\text{\rm qLevi}}(\ell^\infty_\omega(\mathbb{R}))$.  
It follows
$$
   \text{\rm L}_{\text{\rm Levi}}(\ell^\infty_\omega(\mathbb{R}))\subsetneqq
   \text{\rm L}^\sigma_{\text{\rm Levi}}(\ell^\infty_\omega(\mathbb{R})), \ \
   \text{\rm L}_{\text{\rm cqLevi}}(\ell^\infty_\omega(\mathbb{R}))\subsetneqq
   \text{\rm L}^\sigma_{\text{\rm cqLevi}}(\ell^\infty_\omega(\mathbb{R})), 
$$
and $\text{\rm L}_{\text{\rm qLevi}}(\ell^\infty_\omega(\mathbb{R}))\subsetneqq\text{\rm L}^\sigma_{\text{\rm qLevi}}(\ell^\infty_\omega(\mathbb{R}))$. 

\medskip
\noindent
Denote elements of the BL $c$ of convergent real sequences
$c$ by $\mathbf{a}=\sum_{n=1}^\infty a_n \cdot e_n$,
where $e_n$ is the n-th unit vector of $c$ and $a_n$ converges in $\mathbb{R}$.
Since each bounded increasing
net in $c$ is \text{\rm o}-Cauchy, $I_c\in\text{\rm L}_{\text{\rm qLevi}}(c)$.
As, the bounded increasing sequence $I_c\mathbf{f}_n=\mathbf{f}_n=\sum_{k=1}^n e_{2k-1}$ is not 
o-convergent in $c$, $I_c\notin\text{\rm L}^\sigma_{\text{\rm cqLevi}}(c)$.
It follows
$$
   \text{\rm L}_{\text{\rm cqLevi}}(c)\subsetneqq\text{\rm L}_{\text{\rm qLevi}}(c) \ \ \ 
   \text{\rm and} \ \ \ 
   \text{\rm L}^\sigma_{\text{\rm cqLevi}}(c)\subsetneqq\text{\rm L}^\sigma_{\text{\rm qLevi}}(c).
$$
\medskip
\noindent
Consider the embedding operator $J:c\to\ell^\infty$.
It is straightforward $J\in\text{\rm L}_{\text{\rm cqLevi}}(c,\ell^\infty)$.
By considering of the bounded increasing sequence $\sum_{k=1}^n e_{2k}$ in $c$,
we conclude $J\notin\text{\rm L}^\sigma_{\text{\rm cqLevi}}(c,\ell^\infty)$.
It follows
$$
   \text{\rm L}_{\text{\rm Levi}}(c,\ell^\infty)\subsetneqq
   \text{\rm L}_{\text{\rm cqLevi}}(c,\ell^\infty)\ 
   \text{\rm and} \
   \text{\rm L}^\sigma_{\text{\rm Levi}}(c,\ell^\infty)\subsetneqq
   \text{\rm L}^\sigma_{\text{\rm cqLevi}}(c,\ell^\infty).
$$

\medskip
Since every o-Cauchy net (sequence) in a Dedekind ($\sigma$-) complete VL is order convergent, the next proposition
follows directly from Definition \ref{order-to-topology}. 

\begin{proposition}\label{prop1bb}
Let $E$ be a NL and $F$ a Dedekind $\sigma$-complete VL. 
The following holds.
\begin{enumerate}[$a)$]
\item 
$\text{\rm L}^\sigma_{\text{\rm cqLevi}}(E,F)=
\text{\rm L}^\sigma_{\text{\rm qLevi}}(E,F)$.
\item 
$\text{\rm L}_{\text{\rm cqLevi}}(E,F)=
\text{\rm L}_{\text{\rm qLevi}}(E,F)$
if $F$ is Dedekind complete.
\end{enumerate}
\end{proposition}

\noindent
The Dedekind ($\sigma$-) completeness of $F$ is essential in
Proposition \ref{prop1bb}. Indeed, for item $a)$ observe 
that every increasing bounded sequence in $c$ is o-Cauchy, and hence 
$I_c\in\text{\rm L}^\sigma_{\text{\rm qLevi}}(c)$. Pick an increasing 
bounded sequence $\mathbf{x}_n=\sum\limits_{k=1}^n e_{2k}$ 
in $c$. Since $\mathbf{x}_n\convo\sum\limits_{k=1}^\infty e_{2k}$ in $\ell^\infty$
and $\sum\limits_{k=1}^\infty e_{2k}\notin c$ then
$I_c\notin\text{\rm L}^\sigma_{\text{\rm cqLevi}}(c)$.
For item $b)$, consider the Dedekind $\sigma$-complete 
(but not Dedekind complete) BL $c(\mathbb{R})$
such that, for each element $f$ of $c(\mathbb{R})$ there exists some $c_f\in\mathbb{R}$
satisfying $\text{\rm card}\big(\{t\in\mathbb{R}: |f(t)-c_f|\ge 1/n\}\big)<\infty$
for all $n\in\mathbb{N}$. As every increasing bounded net in $c(\mathbb{R})$ 
is o-Cauchy then $I_{c(\mathbb{R})}\in\text{\rm L}_{\text{\rm qLevi}}(c(\mathbb{R}))$.
Since an increasing bounded sequence $(\mathbb{I}_{\{1,2,...,n\}})$ is not o-convergent in $c$ 
then even $I_{c(\mathbb{R})}\notin\text{\rm L}^\sigma_{\text{\rm cqLevi}}(c(\mathbb{R}))$.

\medskip
Like compact operators, the Levi operators extend finite rank continuous operators. More precisely, we have the following result that 
can be viewed as a Levi version of \cite[Lemma 1~{\it iv})]{E2025-1}.

\begin{lemma}\label{prop1}
$\text{\rm F}(E,F)\subseteq\text{\rm L}_{\text{\rm Levi}}(E,F)$
holds for all normed lattices $E$ and $F$.
\end{lemma}

\begin{proof}
Let $T = \sum_{k=1}^n f_k \otimes y_k 
   \  \text{for} \   y_1, \dots, y_n \in TE$
and $f_1, \dots, f_n \in E'$. Denote
$$
   T_1:=\sum_{k=1}^n f_k^+ \otimes y_k \ \ \text{\rm and} \ \ 
   T_2:=\sum_{k=1}^n f_k^- \otimes y_k. 
$$
Let $(x_\alpha)$ be an increasing net in $(B_E)_+$. Then $f_k^+(x_\alpha)$ 
and $f_k^-(x_\alpha)$ are increasing and bounded in $\mathbb{R}$, and hence 
$f_k^+(x_\alpha)\to a_k$ and $f_k^-(x_\alpha)\to b_k$
for some $a_k,b_k\in\mathbb{R}_+$. Since $\dim(TE)<\infty$, it follows
$T_1x_\alpha\convo\sum_{k=1}^n a_k y_k$ and $T_2x_\alpha\convo\sum_{k=1}^n b_k y_k$.
Thus,
$Tx_\alpha=(T_1x_\alpha-T_2x_\alpha)\convo\sum_{k=1}^n(a_k-b_k)y_k\in TE$.
Take $x\in E$ with $Tx=\sum_{k=1}^n(a_k-b_k)y_k$. Then $Tx_\alpha\convo Tx$,
and hence $T\in\text{\rm L}_{\text{\rm Levi}}(E,F)$.
\end{proof}

\medskip
In general, $\text{\rm L}_{\text{\rm Levi}}(E,F)$ is not closed under the addition (see Example \ref{Levi are not VS} below).
The situation with (completely) quasi Levi operators is more pleasant. They 
are closed under linear operations. 

\begin{lemma}\label{prop2}
{\em
Let $E$ be a NL and $F$ be a VL. The following holds.
\begin{enumerate}[$a)$]
\item 
The sets $\text{\rm L}_{\text{\rm cqLevi}}(E,F)$,
$\text{\rm L}^\sigma_{\text{\rm cqLevi}}(E,F)$,
$\text{\rm L}_{\text{\rm qLevi}}(E,F)$,
and $\text{\rm L}^\sigma_{\text{\rm qLevi}}(E,F)$ are vector spaces.
\item 
If $F$ is a NL then
$\text{\rm span}(\text{\rm K}_+(E,F))\subseteq\text{\rm L}_{\text{\rm cqLevi}}(E,F)$.
\end{enumerate}}
\end{lemma}

\begin{proof}
$a)$\
Let $T_1,T_2\in\text{\rm L}_{\text{\rm cqLevi}}(E,F)$, $a_1,a_2\in\mathbb{R}$, and let 
$(x_\alpha)$ be an increasing bounded net in $E$. Then, there exist 
$y_1,y_2\in F$ satisfying $T_1x_{\alpha_\beta}\convo y_1$ and 
$T_2x_{\alpha_\beta}\convo y_2$, and hence 
$(a_1T_1+a_2T_2)x_{\alpha_\beta}\convo a_1y_1+a_2y_2$.

The remaining cases are similar.
\medskip

$b)$\
Accordingly to $a)$, it suffices to show that each positive compact operator 
is completely quasi Levi.
Let $T\in\text{\rm K}_+(E,F)$ and let $(x_\alpha)$ be an 
increasing net in $(B_E)_+$. Then $(Tx_\alpha)$ has a subnet such that
$Tx_{\alpha_\beta}\convn y\in F$. Since $Tx_\alpha\uparrow$ then $Tx_\alpha\convo y$. 
Therefore $T\in\text{\rm L}_{\text{\rm cqLevi}}(E,F)$.
\end{proof}

\medskip
The following proposition shows that the inclusion
$\text{\rm L}_{\text{\rm Levi}}(E)\subseteq\text{\rm L}_{\text{\rm cqLevi}}(E)$
can be proper even while restricted to positive compact operators.
This proposition also provides a counter-example 
to \cite[Proposition 3.5]{AlEG2022}, where it is faultily claimed that every weakly 
compact positive operator between BLs is a Levi operator.

\begin{proposition}\label{prop1aa}
Let $(\beta_k)$ be a strictly positive null sequence of reals.
Then $T\in\text{\rm L}_{\text{\rm cqLevi}}(c,c_0)\setminus\text{\rm L}^\sigma_{\text{\rm Levi}}(c,c_0)$
for the operator $T:c\to c_0$ defined by 
$$
   T\mathbf{a} = \sum_{k=1}^\infty\beta_k a_ke_k. 
$$
Moreover, $T|_{c_0}\in\text{\rm L}_{\text{\rm cqLevi}}(c_0,c_0)
\setminus\text{\rm L}^\sigma_{\text{\rm Levi}}(c_0,c_0)$.
\end{proposition}

\begin{proof}
Since $T\in\text{\rm K}_+(c,c_0)$ then $T\in\text{\rm L}_{\text{\rm cqLevi}}(c,c_0)$ by Lemma~\ref{prop2}. 
Now, consider $\mathbf{x}_n=\sum\limits_{k=1}^n e_{2k}\uparrow$ in $B_{c}$. Then 
$$
   T\mathbf{x}_n =\sum\limits_{k=1}^n\beta_{2k} e_{2k}\convo
   \sum\limits_{k=1}^{\infty}\beta_{2k} e_{2k}\in c_0. 
$$
There is no element $\mathbf{x}=\sum\limits_{k=1}^\infty x_k e_k\in c$ with
$T\mathbf{x}=\sum\limits_{k=1}^{\infty}\beta_{2k} e_{2k}$.
Indeed, otherwise it must satisfy $x_k=\frac{1+(-1)^k}{2}$ for every $k$ which is an absurd. 
Thus, $T\notin\text{\rm L}^\sigma_{\text{\rm Levi}}(c,c_0)$.

\medskip
The case of $T|_{c_0}$ is similar.
\end{proof}

\section{Main results}

\hspace{4mm}
Here, we present more advanced results on Levi operators.
Let us begin with the question of order boundedness.

\medskip
It is proved in \cite[Theorem 2.2]{ZC2022} that every operator which takes increasing bounded nets
of a BL $E$ to vso-convergent nets of a BL $F$ is order bounded. We do not know whether or not every 
quasi Levi operator on a BL is order bounded. However, such operators are continuous.

\medskip
\begin{theorem}\label{T5c}
Every quasi $\sigma$-Levi operator from a BL to a NL is continuous.
\end{theorem}

\begin{proof}\
Let $T:E\to F$ be a quasi $\sigma$-Levi operator from a BL $E$ to a NL $F$.
On the way to a contradiction, suppose $T$ is not continuous. Then, there exists a 
sequence $(x_n)$ in $B_X$ such that $\|Tx_n\|\to\infty$. Without lost of generality,
suppose $(x_n)\subset(B_X)^+$ and $\|Tx_n\|\ge n2^n$ for every $n$.
Set $u=\sum\limits_{k=1}^\infty 2^{-k}x_k$ and $u_n=\sum\limits_{k=1}^n 2^{-k}x_k$ for each $n\in\mathbb{N}$.
Then $0\le u_n\uparrow$ and $\|u_n\|\le 1$ for all $n$. Since $T\in\text{\rm L}_{\text{\rm qLevi}}(E,F)$, 
$(Tu_n)$ is an o-Cauchy sequence in $F$, and hence the sequence $(Tu_n)$ is order bounded. 
Thus, $(Tu_n)$ is also norm bounded. It is absurd because $\|Tu_{n+1}-Tu_n\|=\|T2^{-(n+1)}x_{n+1}\|\ge n+1$ for all $n$.
\end{proof}

\noindent
In general, even rank-one Levi operator need not to be order continuous. To see this, take any continuous linear functional $f$ on a NL $E$
such that $f$ is not order continuous, and observe that $f\in\text{\rm L}_{\text{\rm Levi}}(E,\mathbb{R})$ by Lemma \ref{prop1}.

\medskip
Clearly, $\text{\rm L}_{\text{\rm Levi}}(E,F)$ and $\text{\rm L}^\sigma_{\text{\rm Levi}}(E,F)$
are clearly closed under the scalar multiplication.
In contrast to Lemma \ref{prop2}, the next example (see also \cite[Example 4]{E2025-1}) shows that even a positive rank one perturbation
of a positive \text{\rm Levi} operator belonging to norm closure of finite rank operators
need not to be \text{\rm Levi}.

\begin{example}\label{Levi are not VS}
{\em
Pick a Banach limit functional $f\in(\ell^\infty)'_+$ and 
define $S,T\in\text{\rm L}_+(\ell^\infty)$:
$$
   S\mathbf{x} =
   \sum_{n=2}^\infty\left(\sum_{k=1}^{n-1}\frac{x_k}{2^k}\right) e_n, \ \ \text{\rm and} \ \ \
   T\mathbf{x} = f(\mathbf{x})e_1,
$$
where $\mathbf{x}=\sum_{n=1}^\infty x_n e_n\in\ell^\infty$. 
Since $S=\sum_{m=1}^\infty S_m$ with 
$S_m\mathbf{x}=\frac{x_m}{2^m}\sum_{i=m+1}^\infty e_i$ and $\|S_m\|=\frac{1}{2^m}$,
the operator $S$ lies in the norm closure of $\text{\rm F}_+(\ell^\infty)$.
Operator $T$ is of rank one, and hence $T\in\text{\rm L}_{\text{\rm Levi}}(\ell^\infty)$ 
by Lemma \ref{prop1}.

Let $(B_{\ell^\infty})_+\ni\mathbf{x}_\alpha\uparrow$.
Then, $S\mathbf{x}_\alpha\uparrow S\mathbf{x}$ for
$\mathbf{x}=\sum_{n=1}^\infty\Big(\sup_\alpha[\mathbf{x}_\alpha]_n\Big)e_n\in\ell^\infty$, 
where $[\mathbf{x}_\alpha]_n$ is the n-th coordinate of $\mathbf{x}_\alpha$.
Thus, $S\mathbf{x}_\alpha\convo S\mathbf{x}$, and hence
$S\in\text{\rm L}_{\text{\rm Levi}}(\ell^\infty)$.

Assume $S+T\in\text{\rm L}^\sigma_{\text{\rm Levi}}(\ell^\infty)$.
Take an increasing bounded sequence 
$
   \mathbf{b}_n=\sum_{k=1}^n e_k+\frac{1}{2}\sum_{k=n+1}^\infty e_k\in\ell^\infty.
$
In this case, 
$$
   T\mathbf{b}_n\equiv\frac{1}{2}e_1 \ \ \text{\rm and} \ \ 
   S\mathbf{b}_n\convo\mathbf{u}=\left(0, \ \frac{1}{2}, \ \frac{1}{2}+\frac{1}{4}, \   \frac{1}{2}+\frac{1}{4}+\frac{1}{8}, \ldots \right).
$$
Thus, $(S+T)\mathbf{b}_n\convo\mathbf{v}$ for
$
   \mathbf{v}=\left(\frac{1}{2}, \ \frac{1}{2}, \ \frac{1}{2}+\frac{1}{4}, \ 
   \frac{1}{2}+\frac{1}{4}+\frac{1}{8}, \ldots \right).
$
By the assumption, $(S+T)\mathbf{b}_n\convo(S+T)\mathbf{a}$
for some $\mathbf{a}\in\ell^\infty$. Then 
$
   S\mathbf{b}_n+\frac{1}{2}e_1\convo S\mathbf{a}+T\mathbf{a}= 
   S\mathbf{a}+f(\mathbf{a})e_1.
$
As $[S\mathbf{y}]_1\equiv 0$ then
$f(\mathbf{a})=\frac{1}{2}$.
Since $S\mathbf{b}_n\convo\mathbf{u}$ 
then $S\mathbf{a}=\mathbf{u}$, 
and hence 
$$
   \left(0,\ \frac{1}{2}a_1, \ \frac{1}{2}a_1+\frac{1}{4}a_2, \
   \frac{1}{2}a_1+\frac{1}{4}a_2+\frac{1}{8}a_3, \ldots\right)=
$$
$$
   \left(0, \ \frac{1}{2}, \ \frac{1}{2}+\frac{1}{4}, \ 
   \frac{1}{2}+\frac{1}{4}+\frac{1}{8}, \ldots \right).
$$
So, $\mathbf{a}=\left(1,1, 1, 1, \ldots\right)$,
and then $f(\mathbf{a})=1$. A contradiction. Therefore,  
$S+T\notin\text{\rm L}^\sigma_{\text{\rm Levi}}(\ell^\infty)$.

Observe that $S\vee T=S+T$ because $S\in\text{\rm L}_n(\ell^\infty)$ and 
$T\in\text{\rm L}_n^d(\ell^\infty)$. Thus, $\text{\rm L}_{\text{\rm Levi}}(\ell^\infty)$ 
is also not closed under the lattice operations.

It is worth noting that the operators $S$ and $T$ can be also considered as 
operators belonging to $\text{\rm L}(\ell^\infty,c)$.}
\end{example}

\medskip
We do not know where or not the sum of two order continuous Levi operators 
is a Levi operator.

\medskip
Now, we give some relations between Levi operators and KB-spaces (cf. also \cite[Theorem 2,  Corollary 1]{E2025-1}). 

\begin{theorem}\label{prop8}
For a BL $F$ with order continuous norm the following conditions are equivalent. 
\begin{enumerate}[$i)$]
\item 
$F$ is a \text{\rm KB}-space.
\item  
$\text{\rm L}_+(F)=\text{\rm L}_{\text{\rm Levi+}}(F)$.
\item 
$\text{\rm L}_{\text{\rm qLevi+}}(F)=\text{\rm L}_{\text{\rm Levi+}}(F)$.
\item 
$\text{\rm L}_{\text{\rm cqLevi+}}(F)=\text{\rm L}_{\text{\rm Levi+}}(F)$.
\item 
$\text{\rm K}_+(F)\subseteq\text{\rm L}_{\text{\rm Levi}}(F)$.
\end{enumerate}
\end{theorem}

\begin{proof}
$i) \Longrightarrow ii)$\ 
Just observe that, under the assumption, 
$\text{\rm L}_+(F)\subseteq\text{\rm L}_{\text{\rm Levi+}}(F)$.

\medskip
Implications $ii)\Longrightarrow iii)\Longrightarrow iv)$ are obvious.

\medskip
The implication $iv) \Longrightarrow v)$ follows from Lemma \ref{prop2}.

\medskip
$v) \Longrightarrow i)$ \ 
Since the norm in $F$ is order continuous, each Levi operator on $F$ is a KB-operator in the sense of \cite{AlEG2022,GE2022,E2025-1}.
Therefore, $v)$ implies that each positive compact operator on $F$ is KB. Then, $F$ is a \text{\rm KB}-space by \cite[Theorem 2]{E2025-1}.
\end{proof}

\medskip
The following example shows that 
$\text{\rm L}_{\text{\rm Levi}}(c)$ is not complete in the operator norm.

\begin{example}\label{KB-not-compl}
{\em
Consider $T\in\text{\rm L}(c)$, $T\mathbf{x}=\sum_{k=1}^\infty \frac{x_k}{k} e_k$ for 
$\mathbf{x}=\sum_{k=1}^\infty x_k e_k\in c$,
and define a sequence $(T_n)$ in $\text{\rm L}(c)$ by
$T_n\mathbf{x}=\sum_{k=1}^n\frac{x_k}{k} e_k$. Then:
\begin{enumerate}[$a)$]
\item\
Trivially,
$\text{\rm F}\bigl(c_0\bigl)\ni T_n\convn T$.
\item\
By Lemma \ref{prop1}, 
$T_n\in{\rm L}_{\text{\rm Levi}}\bigl(c\bigl)$.
\item\
By Proposition \ref{prop1aa}, 
$T\notin\text{\rm L}^\sigma_{\text{\rm Levi}}\bigl(c\bigl)$.
\item\
By $b)$ and $c)$, both
${\rm L}_{\text{\rm Levi}}\bigl(c\bigl)$ and 
$\text{\rm L}^\sigma_{\text{\rm Levi}}\bigl(c\bigl)$
are not norm-closed in ${\rm L}\bigl(c\bigl)$.
\end{enumerate}
}
\end{example}

\noindent
We do not know whether the (completely) quasi ($\sigma$-){\text{\rm Levi}} operators are norm complete.

\medskip
Quasi Levi operators satisfy the domination property by \cite[Theorem 2.7]{AlEG2022}.
By a direct modification of the proof of \cite[Theorem 2.7]{AlEG2022},
this fact holds true also for quasi $\sigma$-Levi operators.
However, neither Levi nor $\sigma$-Levi operators satisfy the domination property.

\begin{example}\label{domination for Levi fails}
{\em
Define operators $S,T\in\text{\rm L}(c)$ by
$$
   S\mathbf{a}=\sum_{k=1}^\infty \frac{a_k}{2^k} e_k, \ \ \text{\rm and}\ \ \ 
   T\mathbf{a}=
   \sum_{k=1}^\infty\left(\sum_{j=1}^\infty\frac{a_j}{2^j}\right) e_k.
$$
Clearly, $0\le S\le T$. The operator $T$ is of rank one, and hence  
$T\in\text{\rm L}_{\text{\rm Levi}}(c)$ by Lemma \ref{prop1}. 
By Proposition \ref{prop1aa},
$S\notin\text{\rm L}^\sigma_{\text{\rm Levi}}(c)$ and
$S\in\text{\rm L}_{\text{\rm cqLevi}}(c)$.
}
\end{example}

\noindent
We do not know whether the (completely) quasi ($\sigma$-)Levi operators satisfy the domination property.

\bigskip

{
}
\end{document}